\documentclass[11pt]{amsart}
\usepackage{amsmath,amssymb}
\usepackage[colorlinks, 
linkcolor=blue,
anchorcolor=blue, 
citecolor=red]{hyperref}
\newtheorem{theorem}{Theorem}[section]

\newtheorem{Lemma}[theorem]{Lemma}

\begin{document}
	
	\title[Logarithmic-Sobolev inequalities]{Logarithmic Sobolev inequality in manifolds with nonnegative curvature via the ABP method}
	\author{Lingen Lu}
	
	\address{School of Mathematics and Statistics, Hainan University, Haikou, 570228, P. R. China}
	\email{lulingen@hainanu.edu.cn}
	\begin{abstract}
		In this paper, we employ the ABP method developed by Brendle \cite{Brendle2} to establish the optimal $L^p$ logarithmic Sobolev inequality on manifolds with nonnegative Ricci curvature, as well as a sharp $L^2$ logarithmic Sobolev inequality for submanifolds in manifolds with nonnegative sectional curvature. The sharp constants in both inequalities depend on the asymptotic volume ratio of the ambient manifold.
	\end{abstract}
	
	\maketitle 
	
	\section{Introduction} 
	The classical logarithmic Sobolev inequality on $\mathbb{R}^n$, which was established by Stam \cite{Stam} and Federbush \cite{Federbush}, and later highlighted by Gross \cite{Gross} through its equivalence to the hypercontractivity of a diffusion semigroup, is stated as:	
	\[
		\int_{\mathbb{R}^n} f^2\log f^2\, d\gamma-\bigg(\int_{\mathbb{R}^n} f^2\, d\gamma\bigg)\log\bigg(\int_{\mathbb{R}^n} f^2\, d\gamma\bigg) \le2\int_{\mathbb{R}^n}|Df|^2\, d\gamma,
	\]
	where $d\gamma=(2\pi)^{-\frac n2}e^{-\frac{|x|^2}{2}}\, dx$ is the standard Gaussian probability measure and $f\in C_c^\infty(\mathbb{R}^n)$. This is equivalent to the  optimal $L^2$ logarithmic Sobolev inequality established by Weissler \cite{Weissler}: 
	\[
		\int_{\mathbb{R}^n} f^2\log f^2\, dx\le\frac n2\log\bigg(\frac{2}{\pi e n}\int_{\mathbb{R}^n}|Df|^2\, dx\bigg),
	\]
	valid for $f\in C_c^\infty(\mathbb{R}^n)$ satisfying $\int_{\mathbb{R}^n}f^2\, dx=1$.
	
	The following Euclidean $L^p$ logarithmic Sobolev inequality, established by Beckner \cite{Beckner} for $p=1$, Del Pino and Dolbeault \cite{DPD} for $1<p<n$ and Gentil \cite{Gentil} for general $p>1$, asserts that for $f\in C_c^\infty(\mathbb{R}^n)$ with $\int_{\mathbb{R}^n}f^p\, dx=1$,   
	\[
		\int_{\mathbb{R}^n} |f|^p\log |f|^p\, dx\le\frac np\log\bigg(\mathrm{L}(n,p)\int_{\mathbb{R}^n}|Df|^p\, dx\bigg),
	\]
	where the sharp constant $\mathrm{L}(n,p)$ is given by
	\[
		\mathrm{L}(n,p)=\left\{
		\begin{array}{ll}
			n^{-1}\omega_n^{-\frac1n} & \text{if } p=1,\\
			\frac pn \Big(\frac{p-1}{e}\Big)^{p-1}\Big(\omega_n \Gamma\Big(\frac nq+1\Big)\Big)^{-\frac pn} & \text{if } p>1,
		\end{array} \right.		
	\]
	with $\Gamma$ the Gamma function, $q=\frac{p}{p-1}$ the conjugate of $p$, and $\omega_n=\frac{\pi^{\frac n2}}{\Gamma(\frac n2+1)}$ the volume of the unit ball in $\mathbb{R}^n$.
	
	Recently, by generalizing the optimal transport approach developed by Cordero-Erausquin, Nazaret and Villani \cite{CENV}, Krist\'aly \cite{Kristaly} obtained the optimal $L^p (p\ge1)$ logarithmic Sobolev inequality for manifolds with nonnegative Ricci curvature. In this paper, inspired by the ABP method developed by Brendle \cite{Brendle2}, our first goal is to provide a ABP-based proof of this inequality.
	
	Let $(M^n,g)$, with $n\ge2$, be a complete noncompact manifold with nonnegative Ricci curvature. Denote the asymptotic volume ratio of $M$ by
	\[
		\theta:=\lim_{r\to\infty}\frac{|B_o(r)|}{\omega_n r^n},
	\]
	where $B_o(r)$ is the metric ball centered at $o\in M$  with radius $r$. The Bishop-Gromov volume comparison theorem implies that $\theta\le1$. If $\theta>0$, then we say that $M$ has Euclidean volume growth.

	The sharp $L^1$ logarithmic Sobolev inequality for manifolds with nonnegative Ricci curvature, proved using the ABP method developed by Brendle \cite{Brendle2}, reads as follows:
	\begin{theorem}\label{thm 1.1}
		Let $(M^n,g)$, with $n\ge2$, be a complete noncompact Riemannian manifold with nonnegative Ricci curvature and Euclidean volume growth. Let $\Omega$ be a compact domain in $M$ with smooth boundary $\partial\Omega$, and let $f$ be a nonnegative smooth function on $\Omega$ with $\int_\Omega f \, d\mathrm{vol}=1$. Then
		\[
			\int_\Omega f\log f \, d\mathrm{vol}\le n\log\bigg(n^{-1}\omega_n^{-\frac1n}\theta^{-\frac1n}\bigg(\int_\Omega |Df|\, d\mathrm{vol}+\int_{\partial\Omega}f \, da\bigg)\bigg),
		\]
		where $\theta$ denotes the asymptotic volume ratio of $M$.
	\end{theorem}
	The sharp $L^p(p>1)$ logarithmic Sobolev inequality for manifolds with nonnegative Ricci curvature, which we established based on the ABP method developed by Brendle \cite{Brendle2}, reads as follows:
	\begin{theorem}\label{thm 1.2}
		Let $(M^n,g)$, with $n\ge2$, be a complete noncompact Riemannian manifold with nonnegative Ricci curvature and Euclidean volume growth. Let $\Omega$ be a compact domain in $M$ with smooth boundary $\partial\Omega$, and let $f$ be a nonnegative smooth function on $\Omega$ vanishing on $\partial\Omega$ such that $\int_\Omega f^p \, d\mathrm{vol}=1$ . For $p>1$, we have
		\[
			\int_\Omega f^p\log f^p \, d\mathrm{vol}\le\frac{n}{p}\log\bigg(\mathrm{L}(n,p)\theta^{-\frac{p}{n}}\int_\Omega |Df|^p\, d\mathrm{vol}\bigg),
		\]
		where $\theta$ denotes the asymptotic volume ratio of $M$.
	\end{theorem}
	We now turn to the logarithmic Sobolev inequality for submanifolds. In \cite{Ecker}, Ecker derived such an inequality valid for submanifolds in Euclidean space. Later in 2020, Brendle \cite{Brendle3} established a sharp logarithmic Sobolev inequality for compact submanifolds in Euclidean space without boundary, employing the ABP techniques developed in \cite{Brendle1}. In 2021, Yi and Zheng \cite{YZ} proved a logarithmic Sobolev inequality for compact boundaryless submanifolds in manifolds with nonnegative sectional curvature. Recently, by developing a noncompact modification of Brendle’s ABP method, Wang, Xia and Zhang \cite{WXZ} obtained the optimal logarithmic Sobolev inequality for non-compact self-shrinkers in Euclidean space. Independently, Balogh and Krist\'aly \cite{BK} established the sharp logarithmic Sobolev inequality for submanifolds in Euclidean space using an optimal transport approach. Inspired by \cite{Brendle3, Brendle2, YZ}, our second goal herein is to provide an ABP-based proof of the optimal logarithmic Sobolev inequality for submanifolds in manifolds with nonnegative sectional curvature.
	
	The following theorem extends Theorem 1.1 of \cite{BK} to submanifolds in manifolds with nonnegative sectional curvature.
	
	\begin{theorem}\label{thm1.3}
		Let $(M^{n+m},g)$, with $n\ge2$ and $m\ge1$, be a complete noncompact  Riemannian manifold with nonnegative sectional curvature and Euclidean volume growth. Let $\Sigma^n$ be a compact submanifold of $M$, possibly with boundary, and  let $f$ be a nonnegative smooth function on $\Sigma$ vanishing on $\partial\Sigma$ with $\int_\Sigma f^2\, d\mathrm{vol}=1$. Then 
		\[
		\int_\Sigma f^2\log f^2\, d\mathrm{vol}\le\frac n2\log\bigg(\frac{2}{\pi e n}\theta^{-\frac 2n}\int_\Sigma |D^\Sigma f|^2+\frac14f^2|H|^2\, d\mathrm{vol}\bigg),
		\]
		where $H$ is the mean curvature vector of $\Sigma$ and $\theta$ denotes the asymptotic volume ratio of $M$.
	\end{theorem}
	
	The ABP method, originating from the work of Alexandrov, Bakelman, and Pucci, is fundamental in PDE theory. Cabr\'e was the first to use the ABP method to provide a simple proof of the classical isoperimetric inequality in $\mathbb{R}^n$ around 1996 (see \cite{cabre2}). Recently, by generalizing the ABP technique developed by Cabré \cite{cabre3}, Brendle \cite{Brendle1} established a sharp isoperimetric inequality for minimal submanifolds in Euclidean space of codimension at most $2$.  For further details regarding the ABP method, we refer the reader to \cite{cabre1, cros, lww, pham, WZ, XZ} and the references therein.
	
	This paper is organized as follows: Section 2 is devoted to the ABP-based proof of Theorem \ref{thm 1.1} and Theorem \ref{thm 1.2}. In Section 3, we give an ABP-based proof of Theorem \ref{thm1.3}.
	\section{The case of domains}
	Throughout this section, we assume that $(M^n,g)$, for $n\ge2$, is a complete noncompact manifold with nonnegative Ricci curvature, and we fix a base point $o \in M$. Let $\Omega\subset M$ be a compact domain with smooth boundary $\partial\Omega$. Let $p\ge1$ and  let $f$ be a smooth nonnegative function on $\Omega$. In the case $p>1$, we assume that $f$ vanishes on $\partial\Omega$.
	
	without loss of generality, we assume hereafter that $\Omega$ is connected. Let $\alpha>0$ be a given parameter, $r_o=\max\{d(o,x):x\in \Omega\}$ and $r>1+\alpha r_0$ is an arbitrary constant. By scaling, we may assume that
	\begin{equation}\label{scale1-1}
	\int_\Omega h^p \log h^p \, d\text{vol}-\frac{1}{\alpha}\int_\Omega |Dh|^p \, d\text{vol}=\frac{r^{p-1}}{\alpha}\int_{\partial\Omega}h^p\, da,
	\end{equation}
	where $h=c(f+1/r)$ for some $c>0$. Since $\Omega$ is connected, we can find a solution of the following Neumann boundary problem
	\begin{equation}\label{neumann-1}
		\left\{
		\begin{aligned}
			&\mathrm{div}(h^pD u)=h^p\log h^p-\frac{1}{\alpha}|D h|^p,& \text{ on }\Omega, \\
			&\langle Du,\nu\rangle=\frac{r^{p-1}}{\alpha},
			& \text{ on }\partial\Omega,
		\end{aligned}
		\right.
	\end{equation}
	where $\nu$ is the outward unit normal vector field. By standard elliptic regularity theory, we know that $u\in C^{2,\gamma}$ for each $0<\gamma<1$.
	
	As in \cite{Brendle2}, we set $U:=\{x \in \Omega\setminus\partial\Omega: |D u(x)|<r^{p-1}/\alpha\}$ and 
	\[
		A_r=\{\bar{x}\in U: ru(x)+\frac{1}{2}{d}
		(x,\exp_{\bar{x}}(rD u(\bar{x})))^2\geq
		ru(\bar{x})+\frac{1}{2}r^2|D u(\bar{x})|^2,\ 
		\forall x\in\Omega\}.
	\]
	Moreover, we define the transport map $\Phi_r:\Omega\to M$ by 
	\[
		\Phi_r(x)=\exp_x(r Du(x)),
	\]
	for all $x \in\Omega$. Note that the map $\Phi_r$ is of class $C^{1,\gamma}$ for each $0<\gamma<1$.
	
	The proofs of the following two lemmas are owed to Brendle, so we omit them here.
	\begin{Lemma}\label{lem2.1}
		The set 
		\[
		\Omega_{r}:=\{\bar{x} \in M: d(x,\bar{x})<r^{p}/\alpha ,\,\forall x \in \Omega\}
		\]
		is contained in $\Phi_r(A_r)$.
	\end{Lemma}
	
	\begin{proof}
		See Lemma 2.2 in \cite{Brendle2}.
	\end{proof}
	
	\begin{Lemma}\label{lem2.2}
		The Jacobian determinant of $\Phi_r$ satisfies 
		\[
			|\det D\Phi_r(x)|\le\bigg(1+r\frac{\Delta u(x)}{n}\bigg)^n
		\]
		for all $x\in A_r$.
	\end{Lemma}
	
	\begin{proof}
		See Corollary 2.6 in \cite{Brendle2}.
	\end{proof}
	
	\begin{Lemma}\label{lem2.3}
		The Jacobian determinant of $\Phi_r$ satisfies 
		\[
			\left\{
			\begin{aligned}
				&|\det D\Phi_r(x)|\le r^ne^{\frac{n}{r}-n} h(x), &p=1,\\
				&e^{-\frac{p}{q}\alpha^{\frac{q}{p}}\frac{d(x,\Phi_r(x))^q}{r^q}}	|\det D\Phi_r(x)|\le r^ne^{\frac{n}{r}-n} h(x)^p, &p>1,
			\end{aligned}	
			\right.		
		\]
		for all $x\in A_r$.
	\end{Lemma}
	
	\begin{proof}
		By (\ref{neumann-1}), we have 
		\[
			\begin{aligned}
			\mathrm{div}(h^pD u)&=h^p\Delta u+p\langle h^{p-1} Dh,Du\rangle\\
			& =h^p\log h^p-\frac{1}{\alpha}|D h|^p.
			\end{aligned}
		\]
		In the case $p=1$, using $|Du(x)|<\frac{1}{\alpha}$ and the Cauchy-Schwarz inequality, one has 
		\[
			\begin{aligned}
			\Delta u&=\log h-\frac{ |D h|}{\alpha h}-\frac{\langle Dh,Du\rangle}{h}\\
			&\le \log h.
			\end{aligned}
		\]
		By Lemma \ref{lem2.2} and the elementary inequality $\lambda\le e^{\lambda-1}$, it follows that
		\[
			\begin{aligned}
			|\det D\Phi_r(x)|&\le \bigg(\frac{1}{r}+\frac{\Delta u(x)}{n}\bigg)^n r^n\\
			&\le r^n e^{\frac{n}{r}+\Delta u(x)-n}\\
			&=r^n e^{\frac{n}{r}-n} h(x).
			\end{aligned}
		\]
		For $p>1$, through the Young's inequality, one has 
		\[
			\begin{aligned}
			\Delta u&=\log h^p-\frac{|Dh|^p}{\alpha h^p}-p\frac{\langle Dh,Du\rangle}{h}\\
			&\le \log h^p+\frac{p}{q}\alpha^{\frac{q}{p}}|Du|^q.
			\end{aligned}
		\]
		Similarly, one can get
		\[
			\begin{aligned}
			|\det D\Phi_r(x)|&\le \bigg(\frac{1}{r}+\frac{\Delta u(x)}{n}\bigg)^n r^n\\
			&\le r^n e^{\frac{n}{r}+\Delta u(x)-n}\\
			&\le r^n e^{\frac{n}{r}-n} e^{\frac{p}{q}\alpha^{\frac{q}{p}}|Du(x)|^q}h(x)^p,
			\end{aligned}
		\]
		which implies 
		\[
			e^{-\frac{p}{q}\alpha^{\frac{q}{p}}\frac{d(x,\Phi_r(x))^q}{r^q}}	|\det D\Phi_r(x)|\le r^ne^{\frac{n}{r}-n} h(x)^p.
		\]
		
	\end{proof}
	
	\begin{proof}[\textbf{Proof of Theorem \ref{thm 1.1}}]
		Let $p=1$. Using the triangle inequality, Lemma \ref{lem2.1} and Lemma \ref{lem2.3}, the formula for change of variables in multiple integrals implies 
		\[
			\begin{aligned}
					\theta\omega_n\Big(\frac{r}{\alpha}-r_o\Big)^n&\le\Big|B_o\Big(\frac{r}{\alpha}-r_o\Big)\Big|\\
					&\le|\{\bar{x} \in M: d(x,\bar{x})<r/\alpha,\,\forall x \in \Omega\}|\\
					&\le\int_{A_r}|D\Phi_r(x)|\, d\text{vol}(x)\\
					&\le r^ne^{\frac{n}{r}-n}\int_\Omega h(x)\, d\text{vol}(x)
			\end{aligned}
		\]
		for all $r>1+\alpha r_o$. Dividing the above formula by $r^n$, we obtain
		\[
			\theta\omega_n(\alpha^{-1}-r_o/r)^{n}\le e^{\frac nr-n}\int_\Omega h \, d\text{vol},
		\]
		which implies
		\[
			n- n/r+\log(\theta\omega_n(\alpha^{-1}-r_o/r)^{n})\le\log\bigg(\int_\Omega h \, d\text{vol}\bigg).
		\]
		Combining this inequality with (\ref{scale1-1}) yields the following scale-invariant inequality:
		\begin{equation*}
			\begin{aligned}
				&\int_\Omega h(\log h+n-n/r+\log(\theta\omega_n(\alpha^{-1}-r_o/r)^{n}))\, d\text{vol}-\frac{1}{\alpha}\int_\Omega |Dh|\, d\text{vol}\\
				&\le\bigg(\int_\Omega h\,d\mathrm{vol}\bigg)\log\bigg(\int_\Omega h\,d\mathrm{vol}\bigg)+\frac{1}{\alpha}\int_{\partial\Omega}h \, da.
			\end{aligned}
		\end{equation*}
		Taking  $h=f+\frac1r$ with $r\to\infty$, one can get 
		\begin{equation}
			\begin{aligned}
				&\int_\Omega f(\log f+n+\log(\theta\omega_n\alpha^{-n}))\, d\text{vol}-\frac{1}{\alpha}\int_\Omega |Df|\, d\text{vol}\\
				&\le\bigg(\int_\Omega f\,d\mathrm{vol}\bigg)\log\bigg(\int_\Omega f\,d\mathrm{vol}\bigg)+\frac{1}{\alpha}\int_{\partial\Omega}f \, da.
			\end{aligned}
		\end{equation}
		Let $\int_\Omega f \, d\text{vol}=1$, we have
		\[
			\int_\Omega f\log f \, d\text{vol}\le-n-\log(\theta\omega_n\alpha^{-n})+\frac{1}{\alpha}\bigg(\int_\Omega |Df|\, d\text{vol}+\int_{\partial\Omega}f \, da\bigg).
		\]
		By minimizing the right hand side with respect to $\alpha>0$, one has
		\[
			\alpha=\frac{1}{n}\bigg(\int_\Omega |Df|\, d\text{vol}+\int_{\partial\Omega}f \, da\bigg).
		\]
		This implies 
		\[
			\int_\Omega f\log f \, d\text{vol}\le n\log\bigg(n^{-1}\omega_n^{-\frac1n}\theta^{-\frac1n}\bigg(\int_\Omega |Df|\, d\text{vol}+\int_{\partial\Omega}f \, da\bigg)\bigg).
		\]
	\end{proof}
	
	\begin{proof}[\textbf{Proof of Theorem \ref{thm 1.2}}]
		Let $p>1$. By the triangle inequality, Lemma \ref{lem2.1} and Lemma \ref{lem2.3}, we find
		\begin{align*}
			&\int_{\Omega_{r}} e^{-\frac{p}{q} \alpha^\frac{q}{p}\frac{(r_o+d(o,\bar{x}))^q}{r^q}} \,d\mathrm{vol}(\bar{x})\\
			&\le\int_{\Omega} e^{-\frac{p}{q} \alpha^\frac{q}{p}\frac{d(x,\Phi_r(x))^q}{r^q}}|\det D\Phi_r(x)|1_{A_r}(x) \,d\mathrm{vol}(x)\\
			&\le r^ne^{\frac{n}{r}-n}\int_\Omega h(x)^p\,d\mathrm{vol}(x).
		\end{align*}
		Note that $B_o\Big(\frac{r^p}{\alpha}-r_o\Big)\subset\Omega_{r}$, the co-area formula implies 
		\[
			r^{-n}e^{n-\frac{n}{r}}\int_0^{\frac{r^p}{\alpha}-r_o} e^{-\frac{p}{q} \alpha^\frac{q}{p}\frac{(r_o+\rho)^q}{r^q}} n\omega_n \theta\rho^{n-1}\,d\rho\le \int_\Omega h(x)^p\,d\mathrm{vol}(x).
		\]
		Combining this inequality with (\ref{scale1-1}) and letting $h=f+\frac1r$ as $r\to\infty$, it follows from (A.1) in \cite{BK} that 
		\begin{equation}
			\begin{aligned}
				&\int_\Omega f^p\bigg(\log f^p+n+\log\theta-\frac nq\log\bigg(\frac{p}{q}\alpha^{\frac{q}{p}}\bigg)+\log\bigg(\omega_n\Gamma\bigg(\frac{n}{q}+1\bigg) \bigg)\bigg)\,d\mathrm{vol}\\
				&\le \frac1\alpha\int_\Omega|Df|^p\,d\mathrm{vol}+\bigg(	\int_\Omega f^p\,d\mathrm{vol}\bigg)\log\bigg(\int_\Omega f^p\,d\mathrm{vol}\bigg).
			\end{aligned}
		\end{equation}
		Let $\int_\Omega f^p \, d\text{vol}=1$, we have
		\begin{align*}
			&\int_\Omega f^p\log f^p\,d\mathrm{vol}\\
			&\le -n-\log\theta+\frac nq\log\Big(\frac{p}{q}\alpha^{\frac{q}{p}}\Big)-\log\Big(\omega_n\Gamma\Big(\frac{n}{q}+1\Big) \Big)\bigg)+\frac1\alpha\int_\Omega|Df|^p\,d\mathrm{vol}.
		\end{align*}
		Minimizing the right-hand side with respect to $\alpha>0$ yields 
		\[
			\alpha=\frac pn\int_\Omega|Df|^p\,d\mathrm{vol},
		\]
		which implies 
		\[
			\int_\Omega f^p\log f^p\,d\mathrm{vol}\le\frac{n}{p}\log\bigg(\mathrm{L}(n,p)\theta^{-\frac{p}{n}}\int_\Omega |Df|^p\, d\mathrm{vol}\bigg).
		\]

	\end{proof}
	
	\section{The case of submanifolds}
	In this section, we assume that $(M^{n+m},g)$, with $n\ge2$ and $m\ge1$, is a complete noncompact manifold with nonnegative sectional curvature, and we fix a base point $o \in M$. Let $\Sigma\subset M$ be a compact submanifold of dimension $n$, possibly with boundary. Let $p=2$ and let $f$ be a nonnegative smooth function defined on $\Sigma$ that vanishes on $\partial\Sigma$. Let $\bar{D}$ denote the the Levi-Civita connection of $M$, and let $D^\Sigma$ denote the induced connection on $\Sigma$. The second fundamental form $B$ of $\Sigma$ is defined by 
	\[
		\langle B(X,Y),V\rangle=\langle \bar{D}_XY,V\rangle,
	\]
	where $X$ and $Y$ are tangent vector fields on $\Sigma$, and $V$ is a normal vector field along $\Sigma$. In particular, the mean curvature vector $H$ is defined as the trace of the second fundamental form $B$.
	
	It suffices to prove Theorem \ref{thm1.3} for the case where $\Sigma$ is connected. Let $\alpha>0$ be a given parameter, $r_o=\max\{d(o,x):x\in \Sigma\}$ and $r>1+\alpha r_0$ is an arbitrary constant. By scaling, we can assume that 
	\begin{equation}\label{scale2}
		\int_\Sigma h^2 \log h^2 \, d\mathrm{vol}-\frac{1}{\alpha}\int_\Sigma |D^\Sigma h|^2+\frac{1}{4}h^2|H|^2 \, d\mathrm{vol}=\frac{r}{\alpha}\int_{\partial\Sigma}h^2\, da,
	\end{equation}
	where $h=c(f+1/r)$ for some $c>0$. By the connectedness of $\Sigma$, there exists a solution of the following Neumann boundary problem
	\begin{equation}\label{neumann2}
		\left\{
		\begin{aligned}
			&\mathrm{div}(h^2D^\Sigma u)=h^2\log h^2-\frac{1}{\alpha}|D^\Sigma h|^2-\frac{1}{4\alpha}h^2|H|^2,& \text{ on }\Sigma, \\
			&\langle D^\Sigma u,\nu\rangle=\frac{r}{\alpha},
			& \text{ on }\partial\Sigma,
		\end{aligned}
		\right.
	\end{equation}
	where $\nu$ is the outward unit normal vector field of $\partial\Sigma$ with respect to $\Sigma$. By standard elliptic regularity theory, we know that $u\in C^{2,\gamma}$ for each $0<\gamma<1$.
	
	As in \cite{Brendle2}, we set $U:=\{x \in \Sigma\setminus\partial\Sigma: |D u(x)|<r/\alpha\}$ and 
	\[
		E:=\{(x,y):x\in U,y\in T^\perp_x\Sigma,|D^\Sigma u(x)|^2+|y|^2<r^2/\alpha^2\}.
	\]
	Let $A_r$ be the set of all points $(\bar{x},\bar{y})\in E$ satisfying 
	\[
	ru(x)+\frac{1}{2}{d}(x,\exp_{\bar{x}}(rD^\Sigma u(\bar{x})+r\bar{y}))^2\geq
	ru(\bar{x})+\frac{1}{2}r^2(|D^\Sigma u(\bar{x})|^2+|\bar{y}|^2)
	\]
	for all $x\in\Sigma$. Moreover, we define the transport map $\Phi_r:T^\perp\Sigma\to M$ by 
	\[
	\Phi_r(x,y)=\exp_x(r D^\Sigma u(x)+ry)
	\]
	for all $x \in\Sigma$ and $y\in T^\perp_x\Sigma$. Note that the map $\Phi_r$ is of class $C^{1,\gamma}$ for each $0<\gamma<1$.
	
	The proofs of the following two lemmas are owed to Brendle, so we omit them here.
	
	\begin{Lemma}\label{lem3.1}
		The set 
		\[
			\Sigma_r:=\{\xi\in M: d(x,\xi)<r^2/\alpha \text{ for all } x\in\Sigma\}
		\]
		is contained in $\Phi_r(A_r)$.
	\end{Lemma}
	
	\begin{proof}
		See Lemma 4.2 in \cite{Brendle2}.
	\end{proof}
	
	\begin{Lemma}\label{lem3.2}
		The Jacobian determinant of $\Phi_r$ satisfies 
		\[
			|\det D\Phi_r(x,y)|\le r^{n+m}\bigg(\frac1r+\frac{\Delta_\Sigma u(x)-\langle H(x),y\rangle}{n}\bigg)^n
		\]
		for all $(x,y)\in A_r$.
	\end{Lemma}
	
	\begin{proof}
		See Corollary 4.7 in \cite{Brendle2}.
	\end{proof}
	
	\begin{Lemma}\label{lem3.3}
		The Jacobian determinant of $\Phi_r$ satisfies 
		\[
			e^{-\alpha\frac{d(x,\Phi_r(x,y))^2}{r^2}}|\det D\Phi_r(x,y)|\le r^{n+m}e^{\frac nr-n}e^{-\big|\sqrt{\alpha} y+\frac{1}{2\sqrt{\alpha}}H\big|^2}h(x)^2
		\]
		for all $(x,y)\in A_r$.
	\end{Lemma}
	
	\begin{proof}
		By (\ref{neumann2}), we have 
		\begin{align*}
		\mathrm{div}(h^2D^\Sigma u)&=h^2\Delta_\Sigma u+2\langle hD^\Sigma h,D^\Sigma u\rangle \\
		&=h^2\log h^2-\frac{1}{\alpha}|D^\Sigma h|^2-\frac{1}{4\alpha}h^2|H|^2.
		\end{align*}
		Using the Cauchy-Schwarz inequality, one can derive that
		\begin{align*}
			&\Delta_\Sigma u(x)-\langle H(x),y\rangle\\
			&=\log h(x)^2-\frac{|D^\Sigma h(x)|^2}{\alpha h(x)^2}-\frac{1}{4\alpha}|H(x)|^2-2\frac{\langle D^\Sigma h(x),D^\Sigma u(x)\rangle}{h(x)}-\langle H(x),y\rangle\\
			&\le\log h(x)^2+\alpha|D^\Sigma u(x)|^2-\frac{1}{4\alpha}|H(x)|^2-\langle H(x),y\rangle\\
			&=\log h(x)^2+\alpha(|D^\Sigma u(x)|^2+y^2)-\Big|\sqrt{\alpha}y+\frac{1}{2\sqrt{\alpha}}H(x)\Big|^2,
		\end{align*}
		By Lemma \ref{lem3.2} and the elementary inequality $\lambda\le e^{\lambda-1}$, it follows that
		\[
			e^{-\alpha\frac{d(x,\Phi_r(x,y))^2}{r^2}}|\det D\Phi_r(x,y)|\le r^{n+m}e^{\frac nr-n}e^{-\big|\sqrt{\alpha} y+\frac{1}{2\sqrt{\alpha}}H\big|^2}h(x)^2.
		\]
	\end{proof}
	
	\begin{proof}[\textbf{Proof of Theorem \ref{thm1.3}}]
		Combining Lemma \ref{lem3.1} and Lemma \ref{lem3.3} with the formula for change of variables in multiple integrals, we obtain
		\begin{align*}
			&\int_{\Sigma_r} e^{-\alpha\frac{(r_o+d(o,\xi))^2}{r^2}} \, d\mathrm{vol}(\xi)\\
			&\le\int_\Sigma\Big(\int_{T_x^\perp\Sigma}e^{-\alpha\frac{d(x,\Phi_r(x,y))^2}{r^2}} |\det D\Phi_r(x,y)|1_{A_r}(x,y)\, dy\Big)\, d\mathrm{vol}(x)\\
			&\le r^{n+m}e^{\frac nr-n}\int_\Sigma h(x)^2\Big(\int_{T_x^\perp\Sigma}e^{-\big|\sqrt{\alpha}y+\frac{1}{2\sqrt{\alpha}}H(x)\big|^2}\, dy\Big)\, d\mathrm{vol}(x)\\
			&=r^{n+m}e^{\frac nr-n}\Big(\frac{\pi}{\alpha}\Big)^{\frac m2}\int_\Sigma h(x)^2\, d\mathrm{vol}(x).
		\end{align*}
		Note that $B_o\Big(\frac{r^2}{\alpha}-r_o\Big)\subset\Sigma_{r}$, the co-area formula implies 
		\[
			\int_0^{\frac{r^2}{\alpha}-r_o} e^{-\alpha\frac{(r_o+\rho)^2}{r^2}}(n+m)\theta\omega_{n+m}\rho^{n+m-1}\, d\rho\le r^{n+m}e^{\frac nr-n}\Big(\frac{\pi}{\alpha}\Big)^{\frac m2}\int_\Sigma h(x)^2\, d\mathrm{vol}(x).
		\]
		Combining this inequality with (\ref{scale2}) and letting $h=f+\frac1r$ as $r\to\infty$, one can get
		\begin{equation}
			\begin{aligned}
				&\int_\Sigma f^2\Big(\log f^2+n+\frac n2\log\frac{\pi}{\alpha}+\log\theta\Big)\, d\mathrm{vol}-\frac{1}{\alpha}\int_\Sigma |D^\Sigma f|^2 \, d\mathrm{vol}\\
				&-\frac{1}{4\alpha}\int_\Sigma f^2|H|^2 \, d\mathrm{vol}
				\le\Big(\int_\Sigma f^2\, d\mathrm{vol}\Big)\log\Big(\int_\Sigma f^2\, d\mathrm{vol}\Big).
			\end{aligned}
		\end{equation}
		Let $\int_\Sigma f^2 \, d\text{vol}=1$, we find
		\[
			\int_\Sigma f^2\log f^2\, d\mathrm{vol}\le-n-\frac n2\log\frac{\pi}{\alpha}-\log\theta+\frac{1}{\alpha}\int_\Sigma |Df|^2+\frac14f^2|H|^2 \, d\mathrm{vol}.
		\]
		Minimizing the right hand side with respect to $\alpha>0$ yields 
		\[
			\alpha=\frac2n\int_\Sigma |Df|^2+\frac14f^2|H|^2 \, d\mathrm{vol},
		\]
		which implies 
		\[
			\int_\Sigma f^2\log f^2\le\frac n2\log\bigg(\frac{2}{\pi e n}\theta^{-\frac 2n}\int_\Sigma |Df|^2+\frac14f^2|H|^2 \, d\mathrm{vol}\bigg).
		\]
	\end{proof}
	
	\section*{Acknowledgements}
	We are grateful to Professor Frank Morgan for his useful correspondence, and to Professor Yimin Chen for his helpful discussions.

\end{document}